\documentclass[11pt,a4paper]{article}
\usepackage{theorem,enumerate}
\usepackage{amsmath,latexsym,amssymb,amsfonts,environ,tabularx}
\usepackage{eucal}
\usepackage{color}
\usepackage{xcolor}
\usepackage{comment}
\usepackage{xspace}
\usepackage{tikz}
\usetikzlibrary{fadings,patterns,shapes,positioning,decorations,backgrounds,fit,calc}
\newcounter{Scounter}
\theorembodyfont{\normalfont\slshape}
\newtheorem{thm}{Theorem}%[section]

\newtheorem{Thm}{Theorem}

\newtheorem{lem}[thm]{Lemma}

\newtheorem{Q}[thm]{Question}

\newtheorem{claim}{Claim}[section]

\numberwithin{equation}{section}

\makeatletter
\newcommand{\problemtitle}[1]{\gdef\@problemtitle{#1}}% Store problem title
\newcommand{\probleminput}[1]{\gdef\@probleminput{#1}}% Store problem input
\newcommand{\problemquestion}[1]{\gdef\@problemquestion{#1}}% Store problem question
\NewEnviron{problem}{
	\problemtitle{}\probleminput{}\problemquestion{}% Default input is empty
	\BODY% Parse input
	\par\addvspace{.5\baselineskip}
	\noindent
	\begin{tabularx}{\textwidth}{@{\hspace{\parindent}} l X c}
		\multicolumn{2}{@{\hspace{\parindent}}l}{\@problemtitle} \\% Title
		\textbf{Input:} & \@probleminput \\% Input
		\textbf{Question:} & \@problemquestion% Question
	\end{tabularx}
	\par\addvspace{.5\baselineskip}
}

\newcommand{\qed}{{$\quad\square$\vs{3.6}}}

\newcommand{\vs}[1]{\vspace*{#1 mm}}

\def\C{{ \mathcal{C}}}

\def\S{{ \mathcal{S}}}

\def\Gfc{{ G_f^c}}
\def\bGHc{{ \bar{G}_H^c}}
\def\NP{{ \mathsf{NP}}}
\def\Poly{{ \mathsf{P}}}
\def\odd{{ \text{odd}}}

\definecolor{myorange}{HTML}{E69F00}
\definecolor{myblue}{HTML}{56B4E9}
\definecolor{mygreen}{HTML}{009E73}

\title{A necessary and sufficient condition for the existence\\ of a
properly coloured $f$-factor in an edge-coloured graph}
\author{
	Roman \v{C}ada$^1$\\ {\small \texttt{cadar@kma.zcu.cz}} \and\
	Michitaka Furuya$^2$\\ {\small \texttt{michitaka.furuya@gmail.com}}
	\and\ Kenji Kimura$^3$\\{\small \texttt{kimura@isenshu-u.ac.jp}}
	\and\
	Kenta Ozeki$^4$\\{\small \texttt{ozeki-kenta-xr@ynu.ac.jp}}
	\and
	Christopher Purcell$^1$\footnote{Corresponding author.}\\{\small \texttt{purcell@ntis.zcu.cz}} \and\
	Takamasa
	Yashima$^{1,5}$\\{\small \texttt{takamasa.yashima@gmail.com}}
	\vs{8}\\
	$^1$\textsl{\small European Centre of Excellence NTIS, Department of
	Mathematics,} \\ \textsl{\small University of West Bohemia,}\\
	\textsl{\small P.O. Box 314, 306 14 Pilsen, Czech Republic}\\
	$^2$\textsl{\small College of Liberal Arts and Sciences,} \\
	\textsl{\small Kitasato University,}\\
	\textsl{\small 1-15-1 Kitasato, Minami-ku, Sagamihara, Kanagawa 252-0373, Japan}\\
	$^3$\textsl{\small Department of Information Technology and Electronics,} \\
	\textsl{\small Ishinomaki Senshu University,}\\
	\textsl{\small 1 Shinmito, Minamizakai, Ishinomaki, Miyagi 986-8580, Japan}\\
	$^4$\textsl{\small Faculty of Environment and Information Sciences,} \\
	\textsl{\small Yokohama National University,}\\
	\textsl{\small 79-2 Tokiwadai, Hodogaya-ku, Yokohama, Kanagawa 240-8501, Japan}\\
	$^5$\textsl{\small Department of Computer and Information Science,} \\
	\textsl{\small Seikei Universtiy,}\\
	\textsl{\small 3-3-1 Kichijoji-Kitamachi, Musashino-shi, Tokyo, 180-8633, Japan}
}

\date{}

\begin{document}

\maketitle
\newpage
\begin{abstract}
	The main result of this paper is an edge-coloured version of Tutte's
	$f$-factor theorem. We give a necessary and sufficient condition for
	an edge-coloured graph $G^c$ to have a properly coloured $f$-factor.
	We state and prove our result in terms of an auxiliary graph $G_f^c$
	which has a 1-factor if and only if $G^c$ has a properly coloured
	$f$-factor; this is analogous to the ``short proof'' of the
	$f$-factor theorem given by Tutte in 1954. An alternative statement,
	analogous to the original $f$-factor theorem, is also given. We show
	that our theorem generalises the $f$-factor theorem; that is, the
	former implies the latter. We consider other properties of
	edge-coloured graphs, and show that similar results are unlikely
	for $f$-factors with rainbow components and distance-$d$-coloured
	$f$-factors, even when $d=2$ and the number of colours used is
	asymptotically minimal.
\end{abstract}

\section{Introduction}\label{sec:intro}

In order to characterise the graphs which have a 1-factor, Tutte
proved the following well-known theorem, in which $\odd(H)$ is
the number of components of a graph $H$ with an odd number of
vertices.

\begin{Thm}\label{thm:tutte1}
	\cite{Tutte1947} A graph $G$ has a 1-factor if and only if, for each
	subset $S$ of $V(G)$, $\odd(G-S)\leq |S|$.
\end{Thm}

If we want to characterise the graphs which have an
$f$-factor\footnote{If $G$ is a graph and $f$ is a function
$f:V(G)\rightarrow\mathbb{N}$, an $f$-factor of $G$ is a subgraph $F$
that spans $G$ such that $d_F(v)=f(v)$ for each $v$ in $V(G)$.
Throughout the paper, we include $0$ in $\mathbb{N}$.}, we
can use Theorem~\ref{thm:tutte1} in the following way. For any graph
$G$, we may construct the graph $G_f$ by replacing each vertex $u$ in
$V(G)$ with a complete bipartite graph on $S_u\cup T_u$ where
$S_u=\{u_v\mid uv\in E(G)\}$ and $T_u=\{u'_i\mid 1\leq i \leq
d_G(u){-}f(u)\}$, with an edge in $G_f$ between $u_v$ and $v_u$ whenever
$uv$ is an edge in $G$ (unless $f(v)>d_G(v)$ for some vertex $v$ of $G$,
in which case, $G_f$ does not exist and $G$ can have no $f$-factor). It
is easy to check that $G$ has an $f$-factor if and only if $G_f$
(exists and) has a 1-factor. So we could simply apply
Theorem~\ref{thm:tutte1} to $G_f$ to decide whether $G$ has an
$f$-factor. However, Tutte's $f$-factor theorem, or, more accurately,
its proof as given in \cite{Tutte1954}, says that we need not
check the condition for every subset of $V(G_f)$. For a pair
of disjoint sets $S,T\subseteq V(G)$, let $X_{S,T}= (\bigcup_{u\in S} S_u)
\cup (\bigcup_{u \in T} T_u)$. Tutte's $f$-factor theorem can be stated
as follows.

\begin{Thm}\label{thm:tuttef}
	\cite{Tutte1952} (but see \cite{Tutte1954}) A graph $G$ has an
	$f$-factor if and only if, for each pair of disjoint subsets $S,T$
	of $V(G)$,
	$\odd(G_f-X_{S,T})\leq |X_{S,T}|$.
\end{Thm}

The proof of the $f$-factor theorem given in \cite{Tutte1954} is
essentially a proof of the above theorem followed by a proof that it
is equivalent to the original formulation. In this way, Tutte reduced
$f$-factors to 1-factors, which is not
easy to infer directly from the original (and most common) version of
the theorem given below; here, $h(S,T)$ denotes the number of
components $C$ of $G{-}(S\cup T)$ such that $\sum_{v\in V(C)}
f(v)+|E_G(C,T)|$ is
odd ($E_G(X,Y)$ is the set of edges with one endpoint in $X$ and the
other in $Y$).
%FIXME check this

\begin{Thm}\label{thm:tuttefb}
	\cite{Tutte1952} A graph $G$ has an $f$-factor if and only if, for
	each pair of disjoint subsets $S,T$ of $V(G)$,
	\[
		\sum_{u\in S} f(u) - \sum_{u\in T} (f(u)+d_S(u)) \geq h(S,T).
	\]
\end{Thm}

Our main result is an analogue of the $f$-factor theorem for
properly colored $f$-factors in edge-coloured graphs. Recently,
properly coloured structures in edge-coloured graphs have received
much attention in the literature
\cite{Fujita2017,Fujita2011,Gutin2017,Kano2020,Lo2014a}.

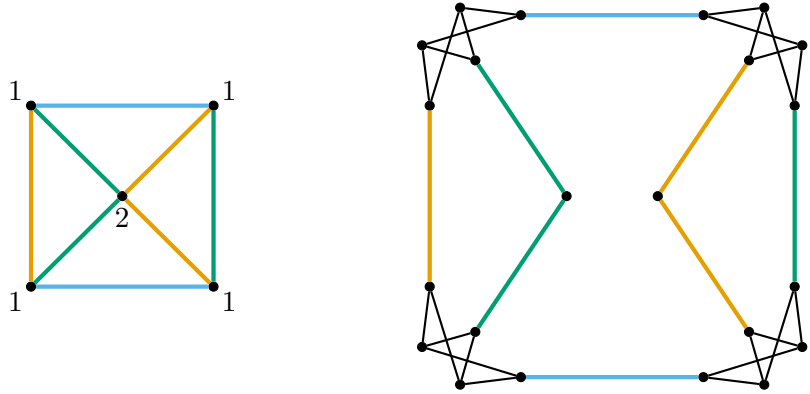
\begin{figure}[t]
	\begin{center}
		\def\e{1.2}
\def\f{0.6}
\def\g{0.7}
\def\h{0.2}
\hspace{1cm}
\begin{tikzpicture}

	\tikzstyle{every node} = [draw,circle,fill=black,inner
	sep=1.2]

	\node[color=white,fill=white] (dummy) at (0,-1.5*\e-\g) {};

	\node[label=below:$2$] (x0) at (0,0) {};
	\node[label=above left:$1$] (x1) at (-\e,\e) {};
	\node[label=above right:$1$] (x2) at (\e,\e) {};
	\node[label=below right:$1$] (x3) at (\e,-\e) {};
	\node[label=below left:$1$] (x4) at (-\e,-\e) {};

	\draw[ultra thick,color=myblue] 
		(x1) -- (x2)
		(x3) -- (x4);
	\draw[ultra thick,color=mygreen]
		(x1) -- (x0) -- (x4)
		(x2) -- (x3);
	\draw[ultra thick,color=myorange]
		(x2) -- (x0) -- (x3)
		(x1) -- (x4);
\end{tikzpicture}
\hspace{2cm}
\begin{tikzpicture}
	\tikzstyle{every node} = [draw,circle,fill=black,inner
	sep=1.2]

	\node (x0g) at (-\e/2,0) {};
	\node (x0r) at (\e/2,0) {};

	\node (x1g) at (-1.5*\e,1.5*\e) {};
	\node (x1b) at (-1.5*\e+\f,1.5*\e+\f) {};
	\node (x1r) at (-1.5*\e-\f,1.5*\e-\f) {};
	\node (x10) at (-1.5*\e-\h,1.5*\e+\g) {};
	\node (x11) at (-1.5*\e-\g,1.5*\e+\h) {};
	
	\foreach \from in {x10,x11}
	\foreach \to in {x1r,x1b,x1g}
	\draw[thick] (\from) -- (\to);

	\node (x2r) at (1.5*\e,1.5*\e) {};
	\node (x2g) at (1.5*\e+\f,1.5*\e-\f) {};
	\node (x2b) at (1.5*\e-\f,1.5*\e+\f) {};
	\node (x20) at (1.5*\e+\h,1.5*\e+\g) {};
	\node (x21) at (1.5*\e+\g,1.5*\e+\h) {};
	\foreach \from in {x20,x21}
	\foreach \to in {x2r,x2b,x2g}
	\draw[thick] (\from) -- (\to);

	\node (x3r) at (1.5*\e,-1.5*\e) {};
	\node (x3g) at (1.5*\e+\f,-1.5*\e+\f) {};
	\node (x3b) at (1.5*\e-\f,-1.5*\e-\f) {};
	\node (x30) at (1.5*\e+\h,-1.5*\e-\g) {};
	\node (x31) at (1.5*\e+\g,-1.5*\e-\h) {};
	\foreach \from in {x30,x31}
	\foreach \to in {x3r,x3b,x3g}
	\draw[thick] (\from) -- (\to);
	
	\node (x4g) at (-1.5*\e,-1.5*\e) {};
	\node (x4b) at (-1.5*\e+\f,-1.5*\e-\f) {};
	\node (x4r) at (-1.5*\e-\f,-1.5*\e+\f) {};
	\node (x40) at (-1.5*\e-\h,-1.5*\e-\g) {};
	\node (x41) at (-1.5*\e-\g,-1.5*\e-\h) {};
	\foreach \from in {x40,x41}
	\foreach \to in {x4r,x4b,x4g}
	\draw[thick] (\from) -- (\to);

	\draw[ultra thick,color=mygreen] (x1g) -- (x0g) -- (x4g);
	\draw[ultra thick,color=myorange] (x2r) -- (x0r) -- (x3r);
	\draw[ultra thick,color=mygreen] (x2g) -- (x3g);
	\draw[ultra thick,color=myorange] (x1r) -- (x4r);
	\draw[ultra thick,color=myblue] 
		(x1b) -- (x2b)
		(x3b) -- (x4b);
\end{tikzpicture}
		\caption{Left: an edge-coloured graph $G^c$ and a function $f$
			from $V(G)$ to $\mathbb{N}$;
		the label on $v$ denotes $f(v)$. Right: the graph $G_f^c$.}
		\label{fig:Gfc}
	\end{center}
\end{figure}

Let $G$ be a graph and let $c:E(G)\rightarrow \{1,\dots,k\}$ be a
colouring of its edges. We denote the resulting edge-coloured graph by
$G^c$. Our aim in this paper is to characterise the pairs $(G^c,f)$, where $f$ is a
function $f:V(G)\rightarrow \mathbb{N}$, such that $G^c$ has a {\em
properly coloured} $f$-factor; that is, one in which adjacent edges
have different colours. We call such a factor a {\em pc-$f$-factor}
for short. In order to decide whether $G^c$ has a pc-$f$-factor we
construct an auxiliary graph in the following way. For each vertex $u$
in $G^c$, we define $E^c(u)$ to be the set of colours assigned by $c$ to the set of edges
incident with $u$. The {\em colour degree} of $u$ is defined
to be $d^c(u)=|E^c(u)|$. We obtain the graph $G^c_f$ from $G^c$ by
replacing each vertex $v$ with a complete bipartite graph on
$S_u\cup T_u$, where $S_u=\{u_i\mid i\in
E^c(u)\}$ and $T_u=\{u'_i\mid 1\leq i \leq d^c(u)-f(u)\}$, with an
edge in $G^c_f$ between $u_j$ and $v_j$ whenever $uv$ is an edge with
$c(uv)=j$; see Figure~\ref{fig:Gfc} for an example of this
construction. It is easy to check that $G^c$ has a pc-$f$-factor if and
only if $\Gfc$ has a 1-factor. To decide whether $G^c$ has such a
factor, we could apply Theorem~\ref{thm:tutte1} to $G^c_f$, similarly
to the uncoloured version described above. Our main result is that we
need not check the condition for every subset, analogously to
Theorem~\ref{thm:tuttef}.

We now introduce some concepts that will be important throughout the
paper.
We define a {\em palette
system} $\S$ of an edge-coloured graph $G^c$ to be a triple
$(S,T,W)$ of disjoint subsets of
$V(G^c)$, where $T$ and $W$ are disjoint unions of subsets of $V(G^c)$ defined as follows. Let
$\hat{f}$ be the maximum value of $f$. Let $\C_0$ and $\C_1$ denote the following families of subsets of ${1,\dots,k}$:
\[
  \C_0= \bigcup_{0\leq i \leq \hat{f}-2} \binom{\{1,\dots,k\}}{i}; \hspace{0.5cm} 
\C_1= \bigcup_{1\leq i \leq \hat{f}-1} \binom{\{1,\dots,k\}}{i}.
\]
Let $\{T^A\subseteq V(G^c) \mid A\in \C_0\}$ be a set of disjoint subsets of $V(G^c)$ and let $T$ be their 
union. Let $\{W^A\subseteq V(G^c) \mid A\in \C_1\}$ be a set of disjoint subsets of $V(G^c)$ and let $W$ be their union. 
For each palette system $\S$, we
further define a set $X_\S$ of vertices of the derived graph $G^c_f$.
Let $X_S,X_T,$ and $X_W$ be defined as follows.
\begin{equation*}
	\begin{aligned}
		X_S &= \bigcup_{u\in S} S_u \\
		X_T &=\bigcup_{A\in \C_0}\bigcup_{u\in T^A}(T_u\cup \{u_i\mid i\in A\})\\
		X_W &=\bigcup_{A\in \C_1}\bigcup_{u \in W^A}\{u_i\mid i \in A\}
	\end{aligned}
\end{equation*}
We define $X_\S=X_S\cup X_T \cup X_W$. In Figure~\ref{fig:STW}, we
show how $X_\S$ intersects $S_v\cup T_v$ when $v$ is in $S$, $T$ and
$W$ respectively. We can now state our main result.
%The order of the \C_0 and \C_1 is correct (C_1 corresponds to X_T)

\begin{thm}\label{thm:main}

	Let $G^c$ be an edge-coloured graph and let $f$ be a function from $V(G)$ to 
	$\mathbb{N}$. Then there is a properly coloured
	$f$-factor in $G^c$ if and only if, for every palette system $\S$,
	we have $\odd(G^c_f-X_\S)\leq |X_\S|$.

\end{thm}

We remark that the necessity of the condition of Theorem~\ref{thm:main} follows immediately from
the 1-factor theorem: if $G^c$ has a pc-$f$-factor, then $G^c_f$ has a
1-factor by construction and so $\odd(G^c_f{-}X)\leq |X|$ for all $X\subseteq V(G^c_f)$ by Theorem~\ref{thm:tutte1}. We
prove the sufficiency of the condition in Section~\ref{sec:main}. We
continue this section by describing some additional computational
results, which we prove in Section~\ref{sec:comp}; we complete the
section with some important definitions.

\begin{figure}[t]
	\begin{center}
		\def\sc{0.9}
\begin{tikzpicture}[scale=\sc]
	
	\useasboundingbox (-1.5,-1) rectangle (1.5,2);
	\tikzstyle{every node} = [draw,circle,fill=black,inner
	sep=1.2]

	\fill[color=white,top color=black!30,bottom color=white]
		(-1.5,-1) .. controls (-3,1.2) and (3,1.2) .. (1.5,-1);

	\draw (0,0) ellipse (1.5 and 0.5);
	\draw (0,1.5) ellipse (1.2 and 0.4);

  \node[draw=none,fill=none] at (0,2.5) {$v \in S$};
  \node[draw=none,fill=none] at (-2.25,1.5) {$T_v$};
  \node[draw=none,fill=none] at (-2.25,0) {$S_v$};
  \node[draw=none,fill=none] at (0,0.85) {$+$};

	\node[color=myorange]  (vr) at (-1,0) {};
	\node[color=mygreen] 	 (vg) at (-0.5,0) {};
	\node[color=myblue] 	 (vb) at (0,0) {};
	\node[draw=none,fill=none] (tdots) at (0.75,0) {$\dots$};

	\node (v1) at (-0.6,1.5) {};
	\node (v2) at (-0.1,1.5) {};
	\node[draw=none,fill=none] (sdots) at (0.5,1.5) {$\dots$};

\end{tikzpicture}
\hspace{1cm}
\begin{tikzpicture}[scale=\sc]

	\useasboundingbox (-1.5,-1) rectangle (1.5,2);
	\tikzstyle{every node} = [draw,circle,fill=black,inner
	sep=1.2]

	\fill[color=white,top color=black!30,bottom color=white]
		(-1.5,-1) .. controls (-3.5,3.5) and (5,3) .. (-0.5,-1);

  \node[draw=none,fill=none] at (0,2.5) {$v \in T^{\{\textcolor{myorange}{1},\textcolor{mygreen}{2},\textcolor{myblue}{3}\}}$};
  \node[draw=none,fill=none] at (0,0.85) {$+$};

	\draw (0,0) ellipse (1.5 and 0.5);
	\draw (0,1.5) ellipse (1.2 and 0.4);

	\node[color=myorange] (vr) at (-1,0) {};
	\node[color=mygreen] 	(vg) at (-0.5,0) {};
	\node[color=myblue] 	(vb) at (0,0) {};
	\node[draw=none,fill=none] (tdots) at (0.75,0) {$\dots$};

	\node (v1) at (-0.6,1.5) {};
	\node (v2) at (-0.1,1.5) {};
	\node[draw=none,fill=none] (sdots) at (0.5,1.5) {$\dots$};

\end{tikzpicture}
\hspace{1cm}
\begin{tikzpicture}[scale=\sc]

	\useasboundingbox (-1.5,-1) rectangle (1.5,2);
	\tikzstyle{every node} = [draw,circle,fill=black,inner
	sep=1.2]

	\fill[color=white,top color=black!30,bottom color=white]
		(-1.5,-1) .. controls (-3,1.2) and (3,1.2) .. (-0.5,-1);

	\draw (0,0) ellipse (1.5 and 0.5);
	\draw (0,1.5) ellipse (1.2 and 0.4);

	\node[draw=none,fill=none] at (0,2.5) {$v \in
		W^{\{\textcolor{myorange}{1},\textcolor{mygreen}{2},\textcolor{myblue}{3}\}}$};
  \node[draw=none,fill=none] at (0,0.85) {$+$};

	\node[color=myorange]  (vr) at (-1,0) {};
	\node[color=mygreen] 	 (vg) at (-0.5,0) {};
	\node[color=myblue] 	 (vb) at (0,0) {};
	\node[draw=none,fill=none] (tdots) at (0.75,0) {$\dots$};

	\node (v1) at (-0.6,1.5) {};
	\node (v2) at (-0.1,1.5) {};
	\node[draw=none,fill=none] (sdots) at (0.5,1.5) {$\dots$};

\end{tikzpicture}
		\caption{In $G_f^c$, $X_\S \cap (S_v \cup T_v)$ depends on whether $v$
		is in $S$, $T$ or $W$. Here, $X_\S$ is depicted in grey.}
		\label{fig:STW}
	\end{center}
\end{figure}
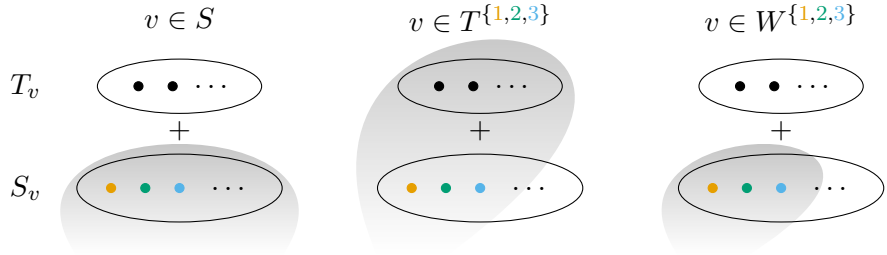

Given an edge-coloured graph $G^c$ and a function $f$ from $V(G^c)$ to
$\mathbb{N}$, we can construct $G_f^c$ efficiently; that is, the
number of computational steps required is bounded by a polynomial in
the number of bits required to represent $G^c$ and $f$ (approximately
$|E(G^c)|{+}|V(G^c)|\log \hat{f}$). Since a 1-factor in a graph can also
be computed efficiently (using any algorithm for maximum matching;
e.g., \cite{Edmonds1965}), we may deduce that finding a 
pc-$f$-factor in an edge-coloured graph, if one exists, can be
done in polynomial time. It is natural to ask whether $f$-factors with
other properties can be reduced to 1-factors efficiently. We say that
an edge-coloured graph is {\em rainbow} if every edge has a
unique colour. Rainbow coloured substructures in edge-coloured graphs
are widely studied; we direct the reader to the survey by Kano and
Xuelinang~\cite{Kano2008}. A {\em rainbow component $f$-factor} (or
rc-$f$-factor) in an edge-coloured graph $G^c$ is an $f$-factor for
which each component is rainbow. We show that deciding
whether an edge-coloured graph has an rc-$f$-factor is
$\mathsf{NP}$-complete, and therefore no efficient reduction from
rc-$f$-factors to 1-factors exists unless $\Poly=\NP$.

Being properly coloured is a local condition while being rainbow
coloured is global. It is natural to ask what role is played by
locality in the difference between pc-$f$-factors and rc-$f$-factors.
We introduce an intermediate case to answer this question. An
edge-coloured graph is {\em distance-$d$-coloured} if the distance
between every monochromatic pair of its edges is at least $d$;
distance-$1$-colouring coincides with proper edge colouring. A
distance-$d$-colouring of a graph $G$ is also a proper vertex
colouring of the $d$th power of the line graph of $G$. Proper vertex
and edge colourings are, of course, widely studied;
distance-$d$-colourings have been the subject of several
papers~\cite{Faudree1989,Ito2007,Molloy1997a}. We show that deciding
the existence of a distance-$d$-coloured $f$-factor in an
edge-coloured graph is $\mathsf{NP}$-complete even for $d=2$. We
strengthen this result by showing that the problem remains hard even
when the number of colours in the graph is asymptotically minimal.

\subsection{Notation and Terminology}

In this paper we consider only finite simple graphs. We follow the
definitions in \cite{Diestel2018} unless stated otherwise and refer the reader to
\cite{Akiyama2011} for a full treatment of factor theory. Here, we
recall some important concepts. To {\em delete} a set $X$ of vertices
from a graph $G$ is to remove them from $V(G)$ and remove all edges
incident with at least one vertex in $X$ from $E(G)$. The resulting
graph is denoted by $G{-}X$. If $X$ is a singleton $\{x\}$, we may
omit the brackets and write $G-x$; we adopt this convention also for
the union and difference of sets and write, say, $(X\setminus x) \cup
y$. If $X$ is a subset of $V(G)$, the subgraph of $G$ {\em induced} by $X$ is
the graph obtained by deleting $V(G)\setminus X$ from $G$.

\section{The Main Result}\label{sec:main}

We start this section by completing the proof of
Theorem~\ref{thm:main}; that is, we prove that the condition in the
theorem is sufficient for an edge-coloured graph to have a
pc-$f$-factor. We then show that our result implies Tutte's $f$-factor
theorem (Theorem~\ref{thm:tuttef}). Finally, we discuss an alternative
statement of the result, analogous to the original statement of the
$f$-factor theorem (Theorem~\ref{thm:tuttefb}).

\subsection{Sufficiency}

In order to show that the condition in Theorem~\ref{thm:main} is
sufficient, completing its proof, we suppose that there exists an
edge-coloured graph $G^c$ which has no pc-$f$-factor for some function
$f$ from $V(G)$ to $\mathbb{N}$. We will show that there exists a
palette system $\S=(S,T,W)$ such that $\odd(G^c_f {-} X_\S)>|X_\S|$,
where $X_\S=X_S\cup X_T \cup X_W$ as defined in Section~1. We know
that $G_f^c$ has no 1-factor by construction; by
Theorem~\ref{thm:tutte1}, there must exist a set of vertices $X$ in
$G_f^c$ such that $\odd(G^c_f{-}X)>|X|$. We say that a set with this
property is \emph{violating}. We want to show that there is a palette
system $\S$ such that $X_\S=X_S \cup X_T \cup X_W$ is violating. In
the following claims, we show that there is a violating set $X$ with
subsets corresponding to $X_S$, $X_T$ and $X_W$. The first claim deals
with the subset corresponding to $X_T$.

\begin{claim}\label{claim:XT}
	Let $v\in V(G^c)$. There exists a violating set $X$ of $\Gfc$ such
	that if $X\cap T_v \not=\emptyset$, then:
	\begin{enumerate}
		\item $T_v \subseteq X$;
		\item $|X\cap S_v| \leq f(v){-}2$.
	\end{enumerate}
\end{claim}

\begin{proof}
	Let $X$ be a violating set of $\Gfc$ with minimum cardinality. We suppose
	that $X\cap T_v$ is not empty, and begin by proving (1). If $S_v\subseteq X$,
	then $\mbox{odd}(G_f^c-(X\setminus T_v))=\mbox{odd}(G_f^c-X)+|X\cap T_v|\ge|X|+1\ge|X\setminus T_v|+1$,
	and so $X\setminus T_v$ is violating, contradicting the minimality of $X$.
	We therefore assume that $S_v\setminus X$ is not empty.
	If
	$f(v)=d^c(v)$, then $T_v=\emptyset$, contradicting the definition of
	$X$; thus, $f(v)<d^c(v)$. Suppose that there exists a vertex $x\in T_v\setminus
	X$. Comparing the components of $G_f^c{-}X$ and $G_f^c{-}(X\setminus
	T_v)$, we see that they only differ at the component containing $x$.
	In particular, we observe that $\odd(G_f^c{-}X) -
	\odd(G_f^c{-}(X\setminus T_v)) \leq 1$. But then we have:
	\begin{align*}
		\odd(G_f^c{-}(X\setminus T_v))   &\geq \odd(\Gfc{-}X)-1\\
																		&\geq |X|\\
																		&=|X\cap T_v|+|X\setminus T_v|\\
																		&\geq |X\setminus T_v| +1,
	\end{align*}
	and therefore $X\setminus T_v$ is violating, contradicting the
	minimality of $X$. We conclude that (1) holds.

	Now suppose that $X\cap T_v\not=\emptyset$ and $|X\cap S_v|\geq
	f(v){-}1$; we will show that (2) holds. By (1), $T_v\subseteq X$. Note that $S_v$
	contributes at most $|S_v\setminus X|$ odd components to $\Gfc{-}X$.
	We observe that $\odd(\Gfc{-}X)$ and $\odd(\Gfc{-}(X\setminus T_v))$
	differ by at most $|S_v\setminus X|$; this maximum is attained when
	each vertex of $S_v\setminus X$ is in a separate odd component and
	$|T_v|+|S_v\setminus X|$ is even. If $|X\cap S_v|=f(v){-}1$,
	then $|T_v|+|S_v \setminus X| = (d^c(v){-}f(v))+(d^c(v){-}f(v){+}1)$ is odd, in
	which case $\odd(\Gfc{-}X)-\odd(\Gfc{-}(X\setminus T_v))\leq
	|S_v\setminus X|{-}1$. Let $\eta=1$ if $|X \cap S_v|=f(v){-}1$ and
	$\eta=0$ otherwise; in other words, we have $|X\cap S_v| \geq
	f(v){-}\eta$. We also have
	$\odd(\Gfc{-}X)-\odd(\Gfc{-}(X\setminus T_v) \leq |S_v \setminus
	X|{-}\eta$. But then we have:
	\begin{align*}
		\odd(\Gfc{-}(X\setminus T_v))
						&\geq \odd(\Gfc{-}X)-|S_v\setminus X|+\eta\\
						&\geq|X|+1-|S_v\setminus X|+\eta\\
						&=|X\setminus T_v|+|X\cap T_v|+1-|S_v\setminus X|+\eta\\
						&=|X\setminus T_v| +(d^c(v)-f(v))+1-(d^c(v)-|X\cap
						S_v|)+\eta\\
						&=|X\setminus T_v|-f(v)+1+|X\cap S_v|+\eta\\
						&\geq|X\setminus T_v|+1,
	\end{align*}
	which contradicts the minimality of $X$ as before; this completes
	the proof of (2) and the claim.
	\hfill\qed

\end{proof}

The next claim shows that a violating set $X$ satisifying the above
claim also has disjoint subsets corresponding to $X_S$ and $X_W$.

\begin{claim}\label{claim:XSW}
	Let $v \in V(G)$. There exists a violating set $X$ of $\Gfc$ which
	satisfies Claim~\ref{claim:XT}, such
	that if $X\cap T_v=\emptyset$ and $X\cap S_v\not=\emptyset$, then
	one of the following holds:
	\begin{enumerate}
		\item $S_v\subseteq X$;
		\item $1\le|X\cap S_v|\leq f(v)-1$.
	\end{enumerate}
\end{claim}

\begin{proof}
	Let $X$ be a violating set satisfying Claim~\ref{claim:XT} and
	suppose $X$ is of maximum cardinality among such sets.
	Further suppose that $X\cap T_v$ is empty but $X\cap S_v$ is not. We
	assume that neither (1) nor (2) hold, and derive a contradiction.
	Since $|S_v|=d^c(v)$, if (1) does not hold then $1\leq|X\cap
	S_v|\leq d^c(v){-}1$. If $d^c(v)=f(v)$, (2) follows immediately,
	giving a contradiction. We may therefore assume that $d^c(v)>f(v)$,
	and thus $T_v$ is not empty. Since (1) does not hold, there exists a
	vertex $u$ in $S_v\setminus X$; note that every vertex in $T_v$ is a neighbour of $u$. Since
	(2) does not hold, $|X \cap S_v| \geq f(v)$. We now consider
	$X'=X\cup S_v$. Observe that each vertex of $T_v$ is isolated, and
	therefore an odd component, in $\Gfc{-} X'$. On the other hand, the
	vertices in $T_v$ are in the same component as $u$ in $\Gfc{-}X$. We
	conclude that $\odd(\Gfc{-}X')$ and $\odd(\Gfc{-}X)$ differ by at
	least $|T_v|{-}1$; this minimum is attained only if
	$|T_v|+|S_v\setminus X|$ is odd. If $|X \cap S_v|=f(v)$, then
	$|T_v|+|S_v \setminus X|=(d^c(v)-f(v))+(d^c(v)-f(v))$ is even, in
	which case $\odd(\Gfc{-}X')-\odd(\Gfc{-}X)\geq |T_v|$. Let $\eta=0$ if
	$|X \cap S_v|=f(v)$ and $\eta=1$ otherwise; in other words, we have
	that $|X \cap S_v| \geq f(v){+}\eta$. We also have that
	$\odd(\Gfc{-}X')-\odd(\Gfc{-}X)\geq |T_v|{-}\eta$. But then we have:
	\begin{align*}
		\odd(\Gfc{-}X')
						&\geq \odd(\Gfc{-}X)+|T_v|-\eta\\
						&\geq|X|+1+|T_v|-\eta\\
						&=|X|+1+|S_v|-f(v)-\eta\\
						&=|X'|+|X\cap S_v|+1-f(v)-\eta\\
						&\ge|X'|+1,
	\end{align*}
	which contradicts the maximality of $X$, completing the proof.
	\hfill\qed
\end{proof}

We can now complete the proof of Theorem~\ref{thm:main}.
Let $X$ be a violating set that satisfies Claims~\ref{claim:XT}
and~\ref{claim:XSW}. It is clear that $X$ is the disjoint union of
three sets corresponding respectively to Claim~\ref{claim:XT} and
the first and second parts of Claim~\ref{claim:XSW}; respectively
and suggestively, we call these sets $T,S$ and $W$. We define these
sets formally as follows:
\begin{align*}
	S   &=\{ u\in V(G^c)\mid S_u\subseteq X \} \\
	T^A   &=\{u\in V(G^c)\mid T_u\subseteq X, S_u\cap
	X=\{u_i\mid i\in A\}\} &(A\in \C_0)\\
	W^A &=\{u\in V(G^c)\mid T_u\cap
	X=\emptyset,S_u\cap X=\{u_i\mid i\in
A\}\} &(A\in \C_1)\\
	T   &=\bigcup_{A\in \C_0} T^A \\
	W   &=\bigcup_{A\in\C_1} W^A.
\end{align*}
Now, with
$X_S,X_T$ and $X_W$ defined as in Section~\ref{sec:intro}, and with
$X_\S=X_S\cup X_T \cup X_W$, it is easy to verify that $X=X_\S$ as
required.

%%%%%%%%%%%%%%%%%%%%%%%%%%%%%%%%%%%%%%%%%%%%%%%%%%%%%%%%%%%%%%%%%%%%%%
%SUBSECTION
%%%%%%%%%%%%%%%%%%%%%%%%%%%%%%%%%%%%%%%%%%%%%%%%%%%%%%%%%%%%%%%%%%%%%%
\subsection{Relationship to Tutte's $f$-Factor Theorem}

We now show that our main result is stronger than the $f$-factor
theorem, in the sense that the latter may be deduced from the former.
For a graph $G$ and a function $f$ from $V(G)$ to $\mathbb{N}$, let
(C1) be the statement that, for all disjoint subsets $S,T$ of $V(G)$,
$\odd(G_f{-}X_{S,T}) \leq |X_{S,T}|$; i.e., (C1) is the condition in
Theorem~\ref{thm:tuttef}. For $G$ and $f$ as before and for a
colouring $c$ of the edges of $G$, let (C2) be the statement that, for
all palette systems $\S$ of $G$, $\odd(\Gfc{-}X_\S) \leq |X_\S|$; i.e.,
(C2) is the condition in Theorem~\ref{thm:main}. Thus,
Theorem~\ref{thm:tuttef} is the statement that $G$ has an $f$-factor
if and only if $G_f$ satisfies (C1); Theorem~\ref{thm:main} is the
statement that $G^c$ has a pc-$f$-factor if and only if $G_f^c$
satisfies (C2).
The following
lemma shows that the restriction of (C2) to proper edge-colourings is
equivalent to (C1).

\begin{lem}\label{lem:pc}
	
	Let $G$ be a graph, let $c$ be a proper edge-colouring of $G$ and
	let $f$ be a function from $V(G)$ to $\mathbb{N}$. Then $G_f$
	satisfies (C1) if and only if $G_f^c$ satisfies (C2).

\end{lem}

\begin{proof}
	Since $c$ is a proper edge-colouring, $G_f$ is isomorphic to $G_f^c$.
	Suppose that $G_f^c$
	satisfies (C2). Let $S$ and $T$ be disjoint subsets of $V(G)$. Let
	$\S$ be the palette system obtained from $S$ and $T$ by leaving
	$S$ unchanged, setting $T^\emptyset = T$, and leaving $T^A$ and
	$W^A$ empty for all non-empty $A$. In this case, it follows from the
	definitions of $X_{S,T}$ and $X_\S$ that $X_{S,T} = X_\S$. Since
	$G_f^c$ satisfies (C2), $\odd(G_f^c{-}X_\S)\leq|X_\S|$, and it
	immediately follows that $\odd(G_f{-}X_{S,T})\leq|X_{S,T}|$; i.e.,
	$G_f$ satisfies (C1).

	Now suppose that $G_f$ satisfies (C1). Let $\S=(S,T,W)$ be a palette
	system in $G^c$. Let $A$ be an element of $\C_1$, and
	consider modifying $\S$ by removing a vertex $w$ from $W^A$. By the
	definition of $X_W$, the effect of this operation is to remove $|A|$
	vertices from $X_W$ and thereby reduce $|X_\S|$ by $|A|$, which is
	at least 1 since $A$ is non-empty. On the other hand, if the component
	of $\Gfc{-}X_\S$ containing $T_w$ is odd and $|A|$ is odd, then the
	effect of removing $w$ from $W^A$ is to
	reduce $\odd(G_f^c{-}X_\S)$ by 1; otherwise, $\odd(\Gfc{-}X_\S)$ stays
	the same or even increases. The effect of modifying
	$\S$ by moving a vertex $v$ from $T^A$ to $T^\emptyset$ is
	identical. We deduce that if $\S'$ is obtained from $\S$ by
	repeatedly performing these operations, then
	$\odd(G_f^c{-}X_{\S'})\leq|X_{\S'}|$ implies
	$\odd(G_f^c{-}X_\S)\leq|X_\S|$. 

	Let $\S^*$ be the palette system obtained from $\S$ by moving all
	vertices in $T^A$ to $T^\emptyset$ and removing all vertices from
	$W^A$, for all non-empty $A$. Observe that $X_{S,T}=X_{\S^*}$. Since
	$G_f$ satisfies (C1), $\odd(G_f{-}X_{S,T})\leq|X_{S,T}|$, from which
	it immediately follows that $\odd(G_f^c{-}X_{\S^*})\leq |X_{\S^*}|$.
	By the argument given in the previous paragraph, we conclude that
	$\odd(G_f^c{-}X_\S)\leq|X_\S|$ as required. \hfill\qed

\end{proof}

\begin{thm}
	Theorem~\ref{thm:main} implies Tutte's $f$-factor theorem.
\end{thm}

\begin{proof}
	Let $G$ be a graph, let $f$ be a function from $V(G)$ to
	$\mathbb{N}$ and let $c$ be a proper edge-colouring of $G$. Suppose
	$G_f$ satisfies (C1) but $G$ has no $f$-factor. Since $G$ has no
	$f$-factor, $G^c$ has no pc-$f$-factor. By Theorem~\ref{thm:main},
	$G_f^c$ does not satisfy (C2). But then $G_f$ does not satisfy (C1)
	by Lemma~\ref{lem:pc}, which is a contradiction.

	Now suppose that $G_f$ does not satisfy (C1) but has an $f$-factor.
	By Lemma~\ref{lem:pc}, $G_f^c$ does not satisfy (C2) and therefore
	by Theorem~\ref{thm:main}, $G^c$ has no pc-$f$-factor. But since $c$
	is a proper edge-colouring, every $f$-factor of $G$ is a
	pc-$f$-factor of $G^c$, so $G$ has no $f$-factor, which is a
	contradiction.\hfill\qed

\end{proof}

\subsection{Alternative Statement}

We now present a different version of Theorem~\ref{thm:main}; the
difference between the original statement and the one that follows is
analogous to the difference between Tutte's $f$-factor theorem as
presented in Theorem~\ref{thm:tuttef} and as in
Theorem~\ref{thm:tuttefb}.

In order to restate our main result, we define two operations on
vertices. Let $G$ be a graph and $c$ a colouring of its edges. 
When we $c$-{\em split} a vertex $v$ of 
$G$, we replace $v$ with $d^c(v)$ vertices denoted by $v_i$ for each
colour $i$ in $E^c(v)$, and for each $w$ in $N_G(v)$, we replace the edge $vw$ by $v_{c(vw)}w$.
When we {\em twin} a vertex $x$ of $G$, we replace $x$ with two
vertices $x_0$ and $x_1$, and for each $y$ in $N_G(x)$, we replace the edge $xy$ with
two edges $x_0y,x_1y$; we also add an edge between $x_0$ and $x_1$.
Let $\S=(S,T,W)$ be a palette system and let $G_\S$ denote the
graph obtained from $G$ by performing the following operations
for every vertex $y$ in $S\cup T \cup W$: if $y$ is in $S$, we delete $y$;
if $y$ is in $T^A$ for some $A$ in $\C_0$, we $c$-split $y$ and then delete
$y_i$ whenever $i\in A$; if $y$ is in $W^A$ for some $A$ in $\C_1$, we delete
the edge $yz$ whenever $c(yz)\in A$ and then twin $y$ if $f(y)+|A|$ is even.
We denote by $h(\S)$ the number of odd components of $G_\S$.

%deleting every vertex in $S$, $c$-splitting
%every vertex in $T$, twinning each vertex $w$ in $W^A$ ($A\in\C_1$) for which
%$f(w){+}|A|$ is even, and deleting the edge $uv$ when $u$ is in $T^A$ ($A\in \C_0$)
%and $c(uv)\in A$ or when $u$ is in $W^A$ ($A\in \C_1$) and $c(uv)\in A$.
%We denote by $h(\S)$ the
%number of odd components of $G_\S$.

\begin{thm}
	Let $G^c$ be an edge-coloured graph and let $f$ be a function from $V(G)$ 
	to $\mathbb{N}$. Then there is a properly coloured
	$f$-factor in $G^c$ if and only if, for every palette system
	$\S$, we have
	\[
		\sum_{u \in S} f(u) + \sum_{A\in \C_0} \sum_{u\in T^A}
		(d^c(u)-f(u)+|A|) + \sum_{A\in \C_1}|W^A| |A| \geq
		h(\S).
	\]
\end{thm}

\begin{proof}
	We show that the inequality in the statement of the theorem is a
	restatement of that in Theorem~\ref{thm:main}.
	Add $\sum_{u\in S} d^c(u)$ to both sides of the inequality in
	the statement of the theorem and rearrange it to obtain:
	\[
		h(\S) + \sum_{u\in S} (d^c(u)-f(u))
		\leq
		\sum_{u\in S} d^c(u) + \sum_{A\in \C_0} \sum_{u\in T^A}
		(d^c(u)-f(u)+|A|) + \sum_{A\in \C_1} \sum_{u \in W^A} |A|.
	\]
	It follows from the definition of $X_\S$ that the right-hand side of
	the (re\-arranged) inequality is equal to $|X_\S|$. By
	Theorem~\ref{thm:main}, it is enough to show that the left-hand
	side is equal to $\odd(\Gfc{-}X_\S)$. 

	Observe that each vertex $v$ in $S$ contributes $d^c(v){-}f(v)$
	isolated vertices (and therefore odd components) to $\Gfc{-}X_\S$.
	We denote those isolated vertices by $I_S$; that is, $I_S=\bigcup_{u
	\in S} T_u$. Consider a vertex $w$ which is
	in $W^A$ for some $A$. Suppose we identify $(S_w \setminus \{w_i
	\mid i \in A\})\cup T_w$ to a single vertex $w^*$ 
	and then, if $f(w){+}|A|$ is even, we twin $w^*$.
	Since
	$|(S_w \setminus \{w_i
	\mid i \in A\})\cup T_w|$ is even if and only
	if
	$f(w){+}|A|$ is even, this operation preserves the number of odd components in the graph.
	If we perform this operation for every vertex in $W,$ we obtain
	from $\Gfc{-}X_\S{-}I_S$ a graph isomorphic to $G_\S$. Therefore,
	$\odd(\Gfc{-}X_\S{-}I_S)=\odd(G_\S)=h(\S)$, which completes the
	proof.	\hfill\qed

\end{proof}

\section{Computational Results}\label{sec:comp}

We now turn to global properties of coloured factors. The main result
of this section is that, given an edge-coloured graph $G^c$ and a
function $f:V(G)\rightarrow \mathbb{N}$, deciding whether $G^c$ has a
rainbow component $f$-factor (rc-$f$-factor) is an $\NP$-complete
problem. This implies that it is not possible to reduce rc-$f$-factors
to 1-factors unless $\Poly{=}\NP$. We omit a full discussion of
computational complexity theory and refer the reader to
\cite{Garey1979}. Formally, the computational problem we are
interested in is the following.

\begin{problem}
	\problemtitle{{\sc Rainbow component factor (rc-fac)}}
	\probleminput{A graph $G$, a colouring $c$ of its edges, a function
	$f:V(G)\rightarrow \mathbb{N}$}
	\problemquestion{Does $G^c$ have a rainbow component $f$-factor?}
\end{problem}
\def\rcfac{{{\sc rc-fac}\xspace}}

To prove our main result we introduce a problem that is known to be
$\NP$-hard. We define a {\em $1$-in-$3$-colouring} of a 3-uniform
hypergraph $H$ to be a function $\phi:V(H)\rightarrow
\{-1,1\}$ such that in each edge of $H$ there is exactly one
vertex $v$ with $\phi(v)=1$.

\begin{problem}
	\problemtitle{{\sc Hypergraph $1$-in-$3$-Colouring (h-1-in-3-col)}}
	\probleminput{A 3-uniform hypergraph $H$}
	\problemquestion{Does $H$ have a $1$-in-$3$-colouring?}
\end{problem}
\def\hc{{{\sc h-1-in-3-col} }}

It is known that \hc is $\NP$-hard even for $k$-regular hypergraphs,
where $k\geq3$~\cite{Kratochvil2003}. We will reduce this version of
\hc to \rcfac{}; given a $k$-regular, 3-uniform hypergraph $H$, we
construct a graph $G_H$, a colouring $c$ and a function $f$ such that
$G_H^c$ has an rc-$f$-factor if and only if $H$ has a
$1$-in-$3$-colouring. In fact, the function $f$ will be the constant
function with $f(v)=k{-}1$ for every vertex $v$ in $V(G)$; thus, we can prove the
stronger result that the following problem is $\NP$-hard, for any
$r\geq 2$.

\begin{problem}
	\problemtitle{{\sc Rainbow component $r$-factor (rc-$r$-fac)}}
	\probleminput{A graph $G$, a colouring $c$ of its edges}
	\problemquestion{Does $G^c$ have a rainbow component $r$-factor?}
\end{problem}

We remark that the constant $r$ is a fixed part of the definition of
the problem, not a part of the input. The hardness of this problem
implies hardness of the version of the problem where the constant is
in the input, which in turn implies hardness of \rcfac{}.

\begin{thm}\label{thm:rcfac} The {\sc rainbow component factor
	problem} is $\NP$-complete. Moreover, the {\sc rainbow component
	$r$-factor} problem is $\NP$-complete, for any integer $r\geq 2$.
\end{thm}
\def\GHc{{G_H^c}}

\begin{proof}
 	It is easy to see that the problem is in $\NP$; a nondeterministic
	Turing machine can guess a set of edges and it is easy to verify in
	polynomial time whether a set of edges induces 1) an $r$-factor and 2)
	a graph each component of which is rainbow (by traversal of the
	components). 

	Let $r$ be an integer greater than or equal to 2, and let
	$H$ be an $(r{+}1)$-regular, 3-uniform hypergraph. We will construct
	an edge-coloured graph $\GHc$ such that $\GHc$ has an rc-$r$-factor
	if and only if $H$ has a $1$-in-$3$-colouring. As discussed above,
	this is enough to prove the theorem.

	We now describe the construction of $\GHc$. For each vertex $x$ in
	$H$, the graph $G_H$ contains a subgraph isomorphic to $K_{r+1}$
	and $r{+}1$ subgraphs isomorphic to $K_r$ (where $K_n$ denotes
	the complete graph on $n$ vertices). We call the former the
	{\em central clique} of $x$, and we denote the latter by $Q^x_i$ for
	$1\leq i \leq r{+}1$. We denote the vertices of the central clique
	of $x$ by $x_1,\dots,x_{r+1}$, and for each $i$ we add an edge from
	$x_i$ to each vertex of $Q^x_i$. 

	For each edge $e=\{x,y,z\}$ in $H$ we add a vertex to $G_H$ denoted
	by $v_e$. We arbitrarily order the edges of $H$ in which each vertex
	appears; we may thus refer to $e$ as (say) the $i$th edge of $x$,
	the $j$th edge of $y$, and the $k$th edge of $z$.
	In that case, we add an edge from $v_e$ to each vertex of $Q^x_i$,
	$Q^y_j$, and $Q^z_k$.

	We now describe the colouring $c$ of the edges of $G_H$. To simplify
	our notation, we introduce the constants $\rho_0=\binom{r}{2}$ and
	$\rho_1 = \binom{r+1}{2}$. All colours assigned by $c$ will be in
	$\{1,\dots, \rho_1\}$. For each vertex $x$ in $H$, $c$ assigns the
	edges of the central clique of $x$ distinct colours (there are
	$\rho_1$ edges in the clique). For $1 \leq i \leq r{+}1$, $c$
	assigns the edges in $Q_i^x$ distinct colours from
	$\{1,\dots,\rho_0\}$ (there are $\rho_0$ such edges). Each of the
	$r$ edges from $x_i$ to $Q_i^x$ is assigned a distinct colour from
	$\{\rho_0{+}1,\dots,\rho_1\}$ by $c$ (note that $\rho_1-\rho_0=r$).
	Finally, when the $j$th edge of $x$ in $H$ is $e$,
	we set $c(v_e u)=c(u x_j)$ for each
	vertex $u$ in $Q_j^x$; in other words, $c$ assigns the edge $v_e u$ the same colour that it assigns
	the edge $u x_j$. It is easy to see that this construction can
	be performed in polynomial time. 

	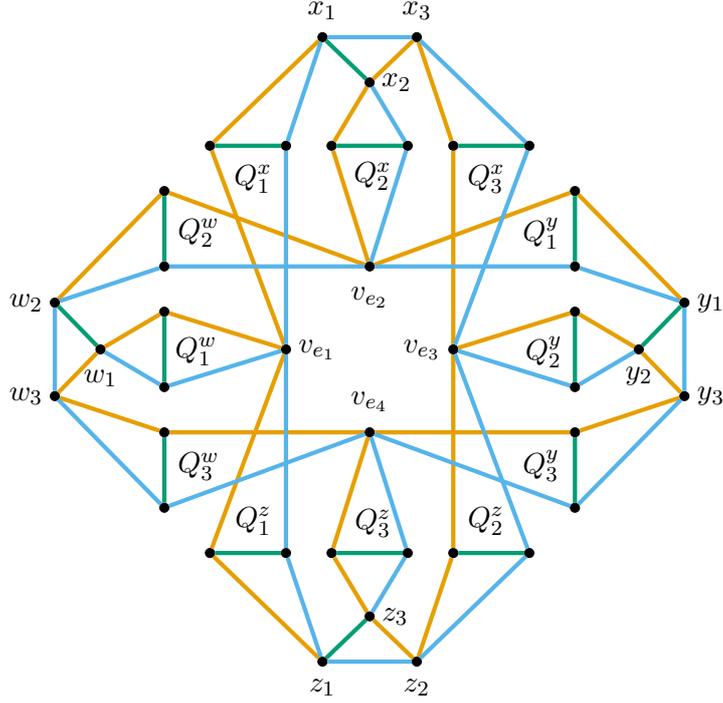
\begin{figure}[t]
		\begin{center}
			\def\a{1.0}%Q edges
\def\b{0.6}%gap between Qs
\def\c{1.2}%central clique to Qs
\def\d{3.0*\a+4.0*\b}%mid square
\begin{tikzpicture}

				\tikzstyle{every node} = [draw,circle,fill=black,inner
				sep=1.2]
				
				%central clique x
				\node(x1)[label=$x_1$] at (\a+1.8*\b+\c,2.2*\c+\d) {};
				\node(x2)[label=0:$x_2$] at (1.5*\a+2*\b+\c,1.7*\c+\d) {};
				\node(x3)[label=$x_3$] at (2*\a+2.2*\b+\c,2.2*\c+\d) {};
				
				%central clique y
				\node(y1)[label=0:$y_1$] at (2.2*\c+\d,2*\a+2.2*\b+\c) {};
				\node(y2)[label=-90:$y_2$] at (1.7*\c+\d,1.5*\a+2*\b+\c) {};
				\node(y3)[label=0:$y_3$] at (2.2*\c+\d,\a+1.8*\b+\c) {};

				%central clique w
				\node(w1)[label=180:$w_2$] at (-0.2*\c,2*\a+2.2*\b+\c) {};
				\node(w2)[label=-90:$w_1$] at (0.3*\c,1.5*\a+2*\b+\c) {};
				\node(w3)[label=180:$w_3$] at (-0.2*\c,\a+1.8*\b+\c) {};

				%central clique z
				\node(z1)[label=-90:$z_1$] at (\a+1.8*\b+\c,-0.2*\c) {};
				\node(z2)[label=0:$z_3$] at (1.5*\a+2*\b+\c,0.3*\c) {};
				\node(z3)[label=-90:$z_2$] at (2*\a+2.2*\b+\c,-0.2*\c) {};

				%Q_x
				\node (x11) at (\b+\c,\c+\d) {};
				\node[label={[label distance=5]-135:$Q_1^x$}] (x12) at (\a+\b+\c,\c+\d) {};
				
				\node (x21) at (\a+2*\b+\c,\c+\d) {};
				\node[label={[label distance=5]-145:$Q_2^x$}] (x22) at (2*\a+2*\b+\c,\c+\d) {};

				\node[label={[label distance=5]-45:$Q_3^x$}] (x31) at (2*\a+3*\b+\c,\c+\d) {};
				\node (x32) at (3*\a+3*\b+\c,\c+\d) {};
				
				%Q_y
				\node (y11) at (\c+\d,3*\a+3*\b+\c) {};
				\node[label={[label distance=5]135:$Q_1^y$}] (y12) at (\c+\d,2*\a+3*\b+\c) {};
				
				\node (y21) at (\c+\d,2*\a+2*\b+\c) {};
				\node[label={[label distance=5]125:$Q_2^y$}] (y22) at (\c+\d,\a+2*\b+\c) {};

				\node[label={[label distance=5]225:$Q_3^y$}] (y31) at (\c+\d,\a+\b+\c) {};
				\node (y32) at (\c+\d,\b+\c) {};

				%Q_z
				\node (z11) at (\b+\c,\c) {};
				\node[label={[label distance=5]135:$Q_1^z$}] (z12) at (\a+\b+\c,\c) {};
				
				\node (z21) at (\a+2*\b+\c,\c) {};
				\node[label={[label distance=5]145:$Q_3^z$}] (z22) at (2*\a+2*\b+\c,\c) {};

				\node[label={[label distance=5]45:$Q_2^z$}] (z31) at (2*\a+3*\b+\c,\c) {};
				\node (z32) at (3*\a+3*\b+\c,\c) {};

				%Q_w
				\node (w11) at (\c,3*\a+3*\b+\c) {};
				\node[label={[label distance=5]45:$Q_2^w$}] (w12) at (\c,2*\a+3*\b+\c) {};
				
				\node (w21) at (\c,2*\a+2*\b+\c) {};
				\node[label={[label distance=5]55:$Q_1^w$}] (w22) at (\c,\a+2*\b+\c) {};

				\node[label={[label distance=5]-45:$Q_3^w$}] (w31) at (\c,\a+\b+\c) {};
				\node (w32) at (\c,\b+\c) {};
				
				%hedges
				\node[label=0:$v_{e_1}$] (e1) at (\a+\b+\c,1.5*\a+2*\b+\c) {};
				\node[label=-90:$v_{e_2}$] (e2) at (1.5*\a+2*\b+\c,2*\a+3*\b+\c)
								{}; 
				\node[label=180:$v_{e_3}$] (e3) at
												(2*\a+3*\b+\c,1.5*\a+2*\b+\c) {};
				\node[label=90:$v_{e_4}$] (e4) at (1.5*\a+2*\b+\c,\a+\b+\c) {};

				%central clique edges
				\foreach \from/\to in {x1/x2,y1/y2,z1/z2,w1/w2}
				\draw[ultra thick,color=mygreen] (\from) -- (\to);
				\foreach \from/\to in {x3/x2,y3/y2,z3/z2,w3/w2}
				\draw[ultra thick,color=myorange] (\from) -- (\to);
				\foreach \from/\to in {x1/x3,y1/y3,z1/z3,w1/w3}
				\draw[ultra thick,color=myblue] (\from) -- (\to);

				%Q-clique edges
				\foreach \from/\to in
				{x11/x12,x21/x22,x31/x32,y11/y12,y21/y22,y31/y32,z11/z12,z21/z22,z31/z32,w11/w12,w21/w22,w31/w32}
				\draw[ultra thick,color=mygreen] (\from) -- (\to);

				%edges from central to Q-cliques
				\foreach \from/\to in {x1/x11,x2/x21,x3/x31,y1/y11,y2/y21,y3/y31,z1/z11,z2/z21,z3/z31,w1/w11,w2/w21,w3/w31}
				\draw[ultra thick,color=myorange] (\from) -- (\to);
				\foreach \from/\to in
				{x1/x12,x2/x22,x3/x32,y1/y12,y2/y22,y3/y32,z1/z12,z2/z22,z3/z32,w1/w12,w2/w22,w3/w32}
				\draw[ultra thick,color=myblue] (\from) -- (\to);

				%edges from hedges to Q-cliques
				\foreach \from/\to in
				{e1/z11,e1/w21,e1/x11,e2/w11,e2/x21,e2/y11,e3/x31,e3/y21,e3/z31,e4/w31,e4/z21,e4/y31}
				\draw[ultra thick,color=myorange] (\from) -- (\to);
				\foreach \from/\to in
				{e1/z12,e1/w22,e1/x12,e2/w12,e2/x22,e2/y12,e3/x32,e3/y22,e3/z32,e4/w32,e4/z22,e4/y32}
				\draw[ultra thick,color=myblue] (\from) -- (\to);
\end{tikzpicture}
			\caption{$G_H^c$ for $r{=}2$ where $H$ is the hypergraph on
				$\{w,x,y,z\}$ with edges $e_1{=}\{z,w,x\}$, $e_2{=}\{w,x,y\}$,
				$e_3{=}\{x,y,z\}$, $e_4{=}\{y,z,w\}$. Colours
				\textcolor{mygreen}{1}, \textcolor{myblue}{2},
				\textcolor{myorange}{3} are shown in \textcolor{mygreen}{{\bf
				green}}, \textcolor{myblue}{{\bf blue}}, and
			\textcolor{myorange}{{\bf orange}}, respectively.}
			\label{fig:GHc}
		\end{center}
	\end{figure}

	In Figure~\ref{fig:GHc} we give an example of the
	construction where $H$ is the complete 3-uniform hypergraph on four
	vertices. It is easy to see that $H$ has no 1-in-3-colouring. The
	reader is invited to check that $G_H^c$ has no rc-2-factor.
	Note that to cover $x_1$, for example, such a factor must include as a component either
	the central clique of $x$ or the clique on
	$x_1\cup Q_1^x$.

	Now suppose that $H$ has a 1-in-3-colouring $\phi$. We construct an
	rc-$r$-factor $F_\phi$ for $\GHc$. Let $x$ be a vertex in $V(H)$ and
	suppose $x$ appears in edges $e_1,\dots,e_{r+1}$. Then, if
	$\phi(x)=1$, we include in $F_\phi$ the edges of the central clique
	of $x$ and the edges of the clique on $v_{e_i}\cup Q_i^x$ for $1\leq
	i\leq r{+}1$; otherwise, we include the edges of the clique on $x_i
	\cup Q_i^x$ for $1\leq i \leq r{+}1$. The obtained subgraph is clearly
	rainbow as every component is a rainbow clique. For each vertex $x$
	in $H$, the vertices in the central clique of $x$ are covered by
	exactly $r$ edges of $F_\phi$, as are those in $Q_i^x$ for $1 \leq i
	\leq r{+}1$. It remains to check that, for each edge $e$ of $H$, the
	vertex $v_e$ is covered by exactly $r$ edges of $F_\phi$. For each
	vertex $x$ in $e$ such that $\phi(x)=1$, we included $r$ edges that
	contain $v_e$. Since $\phi$ is a 1-in-3-colouring, there is exactly
	one such vertex in $e$, and thus $v_e$ is covered by exactly $r$
	edges, as required.

	Now suppose that $\GHc$ has an rc-$r$-factor $F$. We will show that
	$H$ has a 1-in-3 colouring $\phi_F$. Observe that, for each $x$ in
	$V(H)$ and for $1 \leq i \leq r{+}1$, every edge of $Q_i^x$ must be
	in $F$, because every vertex in $Q_i^x$ is incident with $r{+}1$
	edges, two of which have the same colour. Therefore, if there is an
	edge in $F$ from $v_e$ to $Q_i^x$, there cannot be an edge in $F$
	from $v_e$ to $Q_j^y$, unless $x=y$ and $i=j$, because such a pair
	of edges would connect a monochromatic pair of edges. We deduce that
	there is exactly one vertex $x$ in each edge $e$ of $H$ such that
	$v_e\cup Q_i^x$ is a component of $F$ for some value of $i$.
	Furthermore, if $v_e\cup Q_i^x$ is not a component of $F$, then
	$x_i\cup Q_i^x$ is a component of $F$. We conclude that, if the
	$i$th and $j$th edges of $x$ are $e$ and $e'$, and $v_e\cup Q_i^x$
	is a component of $F$, then $v_{e'}\cup Q_j^x$ is a component of $F$
	(otherwise, there is an edge $x_ix_j$ in $F$ and then there are more
	than $r$ edges in $F$ incident with $x_j$). It follows that we may
	define a valid 1-in-3-colouring $\phi_F$ of $H$ by setting
	$\phi_F(x)=1$ if and only if $v_e\cup Q_i^x$ is a component of $F$
	whenever $e$ is the $i$th edge of $x$ in $H$. \hfill\qed

\end{proof}

It is natural to ask whether we can find a nice characterisation of
edge-coloured graphs with $\Pi$-$f$-factors, where $\Pi$ is some local
property of edge-coloured graphs; that is, $\Pi$ is valid for an
edge-coloured graph $G^c$ whenever it is valid within a constant
distance of each vertex. Unfortunately, such a conjecture would be false
in general (unless $\Poly=\NP$). Consider the following problem for
integers $d$ and $r$.

\begin{problem}
	\problemtitle{
		{\sc Distance-$d$-coloured $r$-factor
	(d$d$c-$r$-fac)}} 
	\probleminput{A graph $G$, a colouring $c$ of its edges}
	\problemquestion{Does $G^c$ have a distance-$d$-coloured
	$r$-factor?}

\end{problem}

The proof of Theorem~\ref{thm:rcfac} can be easily modified to show
that the above problem is hard for $r\geq 2$ and $d\geq 3$. It is
enough to demonstrate that, in this case, a set of edges induces a
distance-$d$-coloured $r$-factor if and only if it induces an
rc-$r$-factor. One direction is immediate, since every rc-$r$-factor
in every edge-coloured graph is distance-$d$-coloured by definition.
The other direction requires observing that a distance-$d$-coloured
$r$-factor in the constructed graph $\GHc$ must be the disjoint union
of rainbow cliques. With some modifications to $\GHc$, we may even
prove that {\sc d2c-$r$-fac} is hard using the same proof. However,
the modifications are tedious (we leave them as a puzzle for the
interested reader) and suboptimal in a sense we will now describe. The
colouring constructed in the proof of Theorem~\ref{thm:rcfac} uses
$\Omega(r^2)$ colours. While a distance-$d$-coloured $r$-factor
requires $\Omega(r^2)$ colours for $d\geq 3$, a distance-2-coloured
$r$-factor only requires a number of colours linear in $r$. Therefore,
we omit the modification of the above proof and strengthen the result
for the case $d=2$; that is, we show that {\sc d$2$c-$r$-fac} is hard
even when the number of colours used by the input colouring $c$ is
linear in $r$.

The construction in the proof of the following theorem is similar to
that in the proof of Theorem~\ref{thm:rcfac}, but the role of the
cliques is played by subgraphs isomorphic to the {\em Kneser graph}
$KG(2r{-}1,r{-}1)$. Recall that $[n]$ denotes the set $\{1,\dots,n\}$
and $\binom{[n]}{k}$ denotes the set of subsets of $[n]$ of size $k$.
Then, $KG(n,k)$ is the graph on $\binom{[n]}{k}$ with an edge between
$A$ and $B$ if and only if $A\cap B = \emptyset$. Observe that, in
$KG(2r{-}1,r{-}1)$, if $A$ and $B$ are adjacent then $|A\cup
B|=2r{-}2$; that is, there is a unique element $c^*(AB)$ in
$[2r{-}1]\setminus(A \cup B)$. It is not difficult to check
that $c^*$ is a distance-2-colouring of $KG(2r{-}1,r{-}1)$, which we
call the {\em canonical} colouring; note that $KG(2r{-}1,r{-}1)$ is
$r$-regular. For example, when $r=3$, we have
$KG(5,2)$ which is isomorphic to the Petersen graph, as shown in
Figure~\ref{fig:pet} along with its canonical colouring.

\begin{thm}
	The {\sc distance-$d$-coloured $r$-factor} problem is $\NP$-complete
	for integers $d,r$ at least 2. Furthermore, the {\sc
	distance-2-coloured $r$-factor} problem is $\NP$-complete even when
	the input colouring uses at most $2r{-}1$ colours.
\end{thm}

	\begin{figure}[t]
		\begin{center}
			%1 black 2 grey 3 orange 4 green 5 blue
%FIXME colors are wrong!
\def\rs{1.4}
\def\rrs{3.0}
\begin{tikzpicture}

	\tikzstyle{every node} = [draw,circle,inner	sep=1.2]

	\node (x34) at (18:\rrs)
		{\textbf{\textcolor{myorange}{3}\textcolor{mygreen}{4}}};
	\node (x12) at (90:\rrs)
		{\textbf{\textcolor{black}{1}\textcolor{gray}{2}}};
	\node (x45) at (162:\rrs)
		{\textbf{\textcolor{mygreen}{4}\textcolor{myblue}{5}}};
	\node (x23) at (234:\rrs)
		{\textbf{\textcolor{gray}{2}\textcolor{myorange}{3}}};
	\node (x15) at (306:\rrs)
		{\textbf{\textcolor{black}{1}\textcolor{myblue}{5}}};
	\node (x25) at (18:\rs) 
		{\textbf{\textcolor{gray}{2}\textcolor{myblue}{5}}};
	\node (x35) at (90:\rs)
		{\textbf{\textcolor{myorange}{3}\textcolor{myblue}{5}}};
	\node (x13) at (162:\rs)
		{\textbf{\textcolor{black}{1}\textcolor{myorange}{3}}};
	\node (x14) at (234:\rs)
		{\textbf{\textcolor{black}{1}\textcolor{mygreen}{4}}};
	\node (x24) at (306:\rs)
		{\textbf{\textcolor{gray}{2}\textcolor{mygreen}{4}}};

	%1
	\foreach \from/\to in {23/45,24/35,25/34}
	\draw[ultra thick] (x\from) -- (x\to);

	%2
	\foreach \from/\to in {13/45,14/35,15/34}
	\draw[ultra thick,color=gray] (x\from) -- (x\to);

	%3
	\foreach \from/\to in {12/45,14/25,15/24}
	\draw[ultra thick,color=myorange] (x\from) -- (x\to);

	%4
	\foreach \from/\to in {12/35,13/25,15/23}
	\draw[ultra thick,color=mygreen] (x\from) -- (x\to);

	%5
	\foreach \from/\to in {12/34,13/24,14/23}
	\draw[ultra thick,color=myblue] (x\from) -- (x\to);
\end{tikzpicture}
			\caption{$KG(5,2)$ with its canonical colouring; the vertex
			$\{x,y\}$ is labelled $xy$.}
			\label{fig:pet}
		\end{center}
	\end{figure}
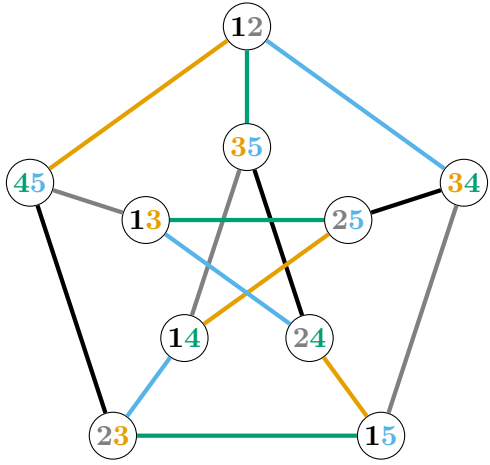

\begin{proof}
	As we discussed above, we need only deal with the second part of the
	statement of the theorem. Furthermore, it is easy to see that the
	problem is in $\NP$ as in the previous theorem. 

	Let $r\geq 2$,
	$\rho=\binom{2r-1}{r-1}$, and let $H$ be a $\rho$-regular, 3-uniform
	hypergraph. We construct a graph $\bar{G}_H$ and a colouring
	$c$ of its edges using at most $2r{-}1$ colours, such that $H$ has a
	1-in-3-colouring if and only if $\bar{G}_H^c$ has a
	distance-2-coloured $r$-factor. 

	We now describe the construction of $\bar{G}_H$. For each vertex $x$
	in $V(H)$ and for $1 \leq i \leq \rho$, we add a set of vertices
	$Q_i^x = \{x_i^X \mid X \in \binom{[2r-1]}{r-1}\}$ and a set of
	vertices $\bar{Q}_i^x = \{\bar{x}_i^X \mid X \in
	\binom{[2r-1]}{r-1}\}$. We add edges so that the respective
	subgraphs induced by $Q_i^x$ and $\bar{Q}_i^x$ are isomorphic to
	$KG(2r{-}1,r{-}1)$; that is, we add an edge $x_i^X x_i^Y$ (and
	$\bar{x}_i^X \bar{x}_i^Y$) if and only if $X\cap Y = \emptyset$. To
	simplify our notation, we will relabel two important vertices. We
	denote $x_i^{\{1,\dots,r-1\}}$ by $x_i$ and
	$\bar{x}_i^{\{r+1,\dots,2r-1\}}$ by $\bar{x}_i$. We add edges
	between the vertices $\{x_1,\dots,x_\rho\}$ so that they induce a
	subgraph isomorphic to $KG(2r{-}1,r{-}1)$; this is done in an
	arbitrary way, and we call the subgraph the {\em central gadget} of
	$x$. We add an edge from $\bar{x}_i$ to each neighbour of $x_i$ in
	$Q_i^x$; that is, we add the edge $\bar{x}_i x_i^X$ if and only if
	$\{1,\dots,r-1\}\cap X = \emptyset$. In this way, $(Q_i^x \setminus
	x_i) \cup \bar{x_i}$ induces a subgraph isomorphic to
	$KG(2r{-}1,r{-}1)$. Finally, we add a vertex $v_e$ for each edge $e$
	of $H$. When $e$ is the $i$th edge of some vertex $x$, we add edges
	from $v_e$ to each neighbour of $\bar{x}_i$ in $\bar{Q}_i^x$; that
	is, we add the edge $v_e \bar{x}_i^X$ if and only if
	$\{r+1,\dots,2r{-}1\}\cap X = \emptyset$. Thus, $(\bar{Q}_i^x
	\setminus \bar{x}_i)\cup v_e$ also induces a subgraph
	isomorphic to $KG(2r{-}1,r{-}1)$. In Figure~\ref{fig:dist}, we have
	depicted part of the constructed graph $\bar{G}_H$ in the case
	$r=3$.

	We now describe the colouring $c$ of the edges of $\bar{G}_H$. The
	central gadget of each vertex is given the canonical colouring
	of $KG(2r{-}1,r{-}1)$. The subgraphs induced by $Q_i^x$ and
	$\bar{Q}_i^x$ are also given the canonical colouring; to
	clarify, we set $c(x_i^X x_i^Y)$ (and $c(\bar{x}_i^X
	\bar{x}_i^Y)$) to be the unique element of $[2r{-}1]\setminus (X
	\cap Y)$. We set
	$c(\bar{x}_i x_i^X)=c(x_i x_i^X)$ and $c(v_e \bar{x}_i^X) =
	c(\bar{x}_i \bar{x}_i^X)$. This completes the description of
	$\bar{G}_H^c$.

	\begin{figure}[t]
		\begin{center}
			\def\r{0.5}
\def\rr{1.2}
\begin{tikzpicture}

				\tikzstyle{every node} = [draw,circle,fill=black,inner sep=1.2]

				\node[label=45:$v_e$] (ve) at (4.5,\rr+0.2) {};
				\begin{scope}[shift={(4.5,\rr+0.2)}]

								\draw[thick] (ve) -- (115:0.37);
								\draw[thick] (ve) -- (135:0.37);
								\draw[thick] (ve) -- (155:0.37);

								\draw[thick] (ve) -- (-25:0.37);
								\draw[thick] (ve) -- (-45:0.37);
								\draw[thick] (ve) -- (-65:0.37);
				\end{scope}

				%Q
				\node (x34) at (18:\rr) {};
				%\node (xi) at (90:\rr) {};
				\node (x45) at (162:\rr) {};
				\node[label={[label distance=12]-25:$Q_i^x$}] (x23) at (234:\rr) {};
				\node (x15) at (306:\rr) {};
				\node (x25) at (18:\r) {};
				\node (x35) at (90:\r) {};
				\node (x13) at (162:\r) {};
				\node (x14) at (234:\r) {};
				\node (x24) at (306:\r) {};

				%Qbar
				\begin{scope}[shift={(3,0)}]
				\node (bx34) at (18:\rr) {};
				%\node (bxi)  at (90:\rr) {};
				\node (bx45) at (162:\rr) {};
				\node[label={[label distance=11]-22:$\bar{Q}_i^x$}] 
								(bx23) at (234:\rr) {};
				\node (bx15) at (306:\rr) {};
				\node (bx25) at (18:\r) {};
				\node (bx35) at (90:\r) {};
				\node (bx13) at (162:\r) {};
				\node (bx14) at (234:\r) {};
				\node (bx24) at (306:\r) {};
				\end{scope}
				\node[label=90:$\bar{x}_i$] (bxi) at (1.5,\rr+0.2) {};

				\begin{scope}[shift={(-1-\rr,\rr+\rr)}]

				%central gadget
				\node (cx34) at (18:\rr) {};
				\node (cxi)  at (90:\rr) {};
				\node (cx45) at (162:\rr) {};
				\node (cx23) at (234:\rr) {};
				\node (cx15) at (306:\rr) {};
				\node[label=-135:$x_i$]
							(xi)	 at (306:\rr) {};
				\node (cx25) at (18:\r) {};
				\node (cx35) at (90:\r) {};
				\node (cx13) at (162:\r) {};
				\node (cx14) at (234:\r) {};
				\node (cx24) at (306:\r) {};

				\draw[thick] (cx34) -- (18:\rr+0.37);
				\draw[thick] (cx34) -- (23:\rr+0.35);
				\draw[thick] (cx34) -- (13:\rr+0.35);

				\draw[thick] (cxi) -- (85:\rr+0.35);
				\draw[thick] (cxi) -- (90:\rr+0.37);
				\draw[thick] (cxi) -- (95:\rr+0.35);

				\draw[thick] (cx45) -- (157:\rr+0.35);
				\draw[thick] (cx45) -- (162:\rr+0.37);
				\draw[thick] (cx45) -- (167:\rr+0.35);

				\draw[thick] (cx23) -- (229:\rr+0.35);
				\draw[thick] (cx23) -- (234:\rr+0.37);
				\draw[thick] (cx23) -- (239:\rr+0.35);

				\draw[thick] (cx25) -- (18:\rr-0.95);
				\draw[thick] (cx25) -- (10:\rr-0.5);
				\draw[thick] (cx25) -- (26:\rr-0.5);

				\draw[thick] (cx35) -- (90:\rr-0.95);
				\draw[thick] (cx35) -- (82:\rr-0.5);
				\draw[thick] (cx35) -- (98:\rr-0.5);

				\draw[thick] (cx13) -- (162:\rr-0.95);
				\draw[thick] (cx13) -- (154:\rr-0.5);
				\draw[thick] (cx13) -- (170:\rr-0.5);
				
				\draw[thick] (cx14) -- (234:\rr-0.95);
				\draw[thick] (cx14) -- (226:\rr-0.5);
				\draw[thick] (cx14) -- (242:\rr-0.5);

				\draw[thick] (cx24) -- (306:\rr-0.95);
				\draw[thick] (cx24) -- (292:\rr-0.5);
				\draw[thick] (cx24) -- (314:\rr-0.5);
				\end{scope}
				
				\foreach \from/\to in
				{i/45,i/35,i/34,34/25,34/15,15/24,15/23,23/14,23/45,45/13,35/14,35/24,25/13,25/14,24/13}{
				\draw[thick] (x\from) -- (x\to);
				\draw[thick] (bx\from) -- (bx\to);
				\draw[thick] (cx\from) -- (cx\to);}

				\draw[thick] (bxi) -- (x45);
				\draw[thick] (bxi) -- (x35);
				\draw[thick] (bxi) -- (x34);

				\draw[thick] (ve) -- (bx45);
				\draw[thick] (ve) -- (bx35);
				\draw[thick] (ve) -- (bx34);

\end{tikzpicture}
			\caption{A subgraph of $\bar{G}_H$ for some
				hypergraph $H$, showing how the central clique
			of $x$ is connected to $v_e$, when $r=3$.}
			\label{fig:dist}
		\end{center}
	\end{figure}
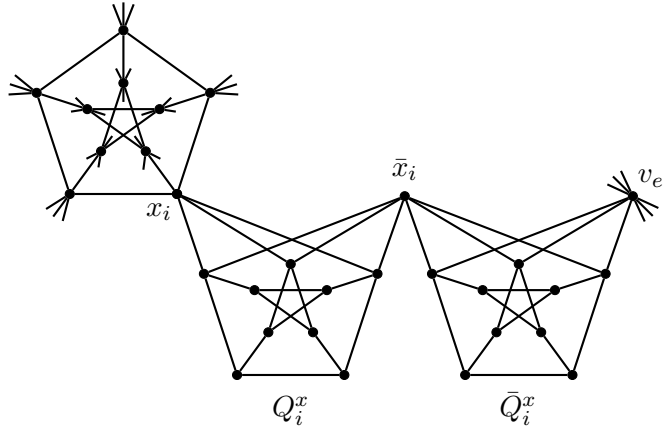

	Suppose that $H$ has a 1-in-3-colouring $\phi$. We construct a
	distance-2-coloured $r$-factor $F_\phi$ for $\bGHc$. Let $x$ be a
	vertex in $V(H)$ that appears in the edges $e_1,\dots,e_\rho$.
	Then, if $\phi(x)=1$, we include in $F_\phi$ the edges of the
	central gadget of $x$, the edges in the subgraph induced by
	$(Q_i^x \setminus x_i)\cup \bar{x}_i$ and the edges in
	the subgraph induced by $(\bar{Q}_i^x \setminus
	\bar{x}_i)\cup v_{e_i}$, for $1 \leq i \leq \rho$. If
	$\phi(x)=-1$, we include in $F_\phi$ the edges in $Q_i^x$ and
	$\bar{Q}_i^x$, for $1\leq i \leq \rho$. Every component of the
	subgraph induced by $F_\phi$ is distance-2-coloured; verifying that $v_e$
	is covered by $r$ edges of $F_\phi$ proceeds as in the
	proof of Theorem~\ref{thm:rcfac}.

	Now suppose that $\bGHc$ has a distance-2-coloured $r$-factor $F$.
	We will show that $H$ has a 1-in-3-colouring $\phi_F$. Observe that,
	for each $x$ in $V(H)$ and for $1 \leq i \leq \rho$, every edge in
	the subgraphs induced by $Q_i^x\setminus x_i$ and
	$\bar{Q}_i^x\setminus\bar{x}_i$ must be in $F$. We claim that there
	cannot be a vertex $u$ in $Q_i^x \setminus x_i$ and a vertex $u'$ in
	$\bar{Q}_i^x \setminus \bar{x}_i$ such that both $u\bar{x}_i$ and
	$u'\bar{x}_i$ are in $F$. From this it follows that for each $x$ in
	$V(H)$, either the central gadget of $x$ is a component of $F$ or
	none of its edges are in $F$ and the rest of the proof is identical
	to that of Theorem~\ref{thm:rcfac}. Suppose the claim is false, and
	observe that $u\bar{x}_i$ has a colour in $\{r,\dots,2r{-}1\}$ and
	$u'\bar{x}_i$ has a colour in $\{1,\dots,r\}$. The edges containing
	$u$ in $Q_i^x\setminus x_i$ have colours $1,\dots,r{-}1$; thus,
	$u'\bar{x}_i$ has colour $r$ otherwise $F$ is not
	distance-2-coloured. The edges containing $u'$ in
	$\bar{Q}_i^x\setminus\bar{x}_i$ have colours $r{+}1,\dots,2r{-}1$;
	thus, $u\bar{x}_i$ has colour $r$ as well, which is a contradiction.
	\hfill\qed

\end{proof}

%%%%%%%%%%%%%%%%%%%%%%%%%%%%%%%%%%%%%%%%%%%%%%%%%%%%%%%%%%%%%%%%%%%%%%%%%%%%%%%%%%%%%%%%%%%%%%%%%%%%%
\section*{Acknowledgments}
%%%%%%%%%%%%%%%%%%%%%%%%%%%%%%%%%%%%%%%%%%%%%%%%%%%%%%%%%%%%%%%%%%%%%%%%%%%%%%%%%%%%%%%%%%%%%%%%%%%%%

This work was supported
by project GA20-09525S of the Czech Science Foundation (to R.~\v{C}ada),
by JSPS Grant-in-Aid for Scientific Research (C) JP23K03204 (to M.~Furuya),
by JSPS Grant-in-Aid for Scientific Research (C) JP23K03195 (to K.~Ozeki),
and by JSPS Grant-in-Aid for Early-Career Scientists JP20K14353 (to T.~Yashima).

\bibliographystyle{plain}
\bibliography{pcf-new}

\begin{thebibliography}{10}

\bibitem{Akiyama2011}
Jin Akiyama and Mikio Kano.
\newblock {\em Factors and {{Factorizations}} of {{Graphs}}: {{Proof
  Techniques}} in {{Factor Theory}}}, volume 2031 of {\em Lecture {{Notes}} in
  {{Mathematics}}}.
\newblock {Springer Berlin Heidelberg}, {Berlin, Heidelberg}, 2011.

\bibitem{Diestel2018}
Reinhard Diestel.
\newblock {\em Graph {{Theory}}}.
\newblock {Springer Berlin Heidelberg}, 2018.

\bibitem{Edmonds1965}
Jack Edmonds.
\newblock Paths, trees, and flowers.
\newblock {\em Canadian Journal of Mathematics}, 17:449--467, 1965/ed.

\bibitem{Faudree1989}
R.J. Faudree, A.~Gy{\'a}rfas, R.H. Schelp, and {\relax Zs}.~Tuza.
\newblock Induced matchings in bipartite graphs.
\newblock {\em Discrete Mathematics}, 78(1-2):83--87, 1989.

\bibitem{Fujita2017}
S.~Fujita, R.~Li, and S.~Zhang.
\newblock Color degree and monochromatic degree conditions for short properly
  colored cycles in edge-colored graphs.
\newblock {\em Journal of Graph Theory}, 87:362--373, 2017.

\bibitem{Fujita2011}
S.~Fujita and C.~Magnant.
\newblock Properly colored paths and cycles.
\newblock {\em Discrete Applied Mathematics}, 159:1391--1397, 2011.

\bibitem{Garey1979}
Michael~R. Garey and David~S. Johnson.
\newblock {\em Computers and Intractability: A Guide to the Theory of
  {{NP-completeness}}}.
\newblock {W. H. Freeman}, 1979.

\bibitem{Gutin2017}
G.~Gutin, B.~Sheng, and M.~Wahlstr{\"o}m.
\newblock Odd properly colored cycles in edge-colored graphs.
\newblock {\em Discrete Mathematics}, 340:817--821, 2017.

\bibitem{Ito2007}
Takehiro Ito, Akira Kato, Xiao Zhou, and Takao Nishizeki.
\newblock Algorithms for finding distance-edge-colorings of graphs.
\newblock {\em Journal of Discrete Algorithms}, 5(2):304--322, 2007.

\bibitem{Kano2020}
M.~Kano, S.~Maezawa, K.~Ota, M.~Tsugaki, and T.~Yashima.
\newblock Color degree sum conditions for properly colored spanning trees in
  edge-colored graphs.
\newblock {\em Discrete Mathematics}, 343:11, 2020.

\bibitem{Kano2008}
Mikio Kano and Li~Xueliang.
\newblock Monochromatic and heterochromatic subgraphs in edge-colored graphs -
  a survey.
\newblock {\em Graphs and Combinatorics}, 24(4):237--263, 2008.

\bibitem{Kratochvil2003}
Jan Kratochv{\'i}l.
\newblock Complexity of hypergraph coloring and {{Seidel}}'s switching.
\newblock In {\em Graph-Theoretic Concepts in Computer Science}, pages
  297--308. {Springer Berlin Heidelberg}, 2003.

\bibitem{Lo2014a}
A.~Lo.
\newblock An edge-colored version of {{Dirac}}'s {{Theorem}}.
\newblock {\em SIAM Journal on Discrete Mathematics}, 28(1):18--36, 2014.

\bibitem{Molloy1997a}
Michael Molloy and Bruce Reed.
\newblock A bound on the strong chromatic index of a graph.
\newblock {\em Journal of Combinatorial Theory, Series B}, 69(2):103--109,
  1997.

\bibitem{Tutte1947}
William~T. Tutte.
\newblock The factorization of linear graphs.
\newblock {\em Journal of the London Mathematical Society}, 22:107--111, 1947.

\bibitem{Tutte1952}
William~T. Tutte.
\newblock The factors of graphs.
\newblock {\em Canadian Journal of Mathematics}, 4:314--328, 1952.

\bibitem{Tutte1954}
William~T. Tutte.
\newblock A short proof of the factor theorem for finite graphs.
\newblock {\em Canadian Journal of Mathematics}, 6:347--352, 1954.

\end{thebibliography}

\end{document}